\documentclass[a4paper,12pt]{amsart}
\usepackage{amssymb}
\usepackage{ifthen}
\usepackage{graphicx}
\usepackage{float}
\usepackage{caption}
\usepackage{subcaption}
\usepackage{cite}
\usepackage{amsfonts}
\usepackage{amscd}
\usepackage{amsxtra}
\usepackage{color}
\usepackage[dvipsnames]{xcolor}
\usepackage[colorlinks, allcolors=RedViolet,pdfstartview=,pdfpagemode=UseNone]{hyperref}

\newtheorem{thm}{Theorem}
\newtheorem{cor}{Corollary}
\newtheorem{lem}{Lemma}
\newtheorem{rem}{Remark}

\newtheorem{conj}{Conjecture}
\newtheorem{prob}{Problem}

\theoremstyle{definition}
\newtheorem{defn}{Definition}[section]
\newtheorem{example}{Example}

\newcounter{alphabet}
\newcounter{tmp}
\newenvironment{Thm}[1][]{\refstepcounter{alphabet}%
\bigskip%
\noindent%
{\bf Theorem \Alph{alphabet}}%
\ifthenelse{\equal{#1}{}}{}{ (#1)}%
{\bf .} \itshape}{\vskip 8pt}

\makeatletter
\renewcommand{\Ref}[1]{\@ifundefined{r@#1}{}{\setcounter{tmp}{\ref{#1}}\Alph{tmp}}}
\makeatother

\newcommand{\es}{{\mathcal S}}

\newcommand{\IC}{{\mathbb C}}
\newcommand{\ID}{{\mathbb D}}

\def\be{\begin{equation}}
\def\ee{\end{equation}}

\newcommand{\bee}{\begin{enumerate}}
\newcommand{\eee}{\end{enumerate}}

\newcommand{\blem}{\begin{lem}}
\newcommand{\elem}{\end{lem}}
\newcommand{\bthm}{\begin{thm}}
\newcommand{\ethm}{\end{thm}}
\newcommand{\bcor}{\begin{cor}}
\newcommand{\ecor}{\end{cor}}
\newcommand{\beg}{\begin{example}}
\newcommand{\eeg}{\end{example}}
\newcommand{\begs}{\begin{examples}}
\newcommand{\eegs}{\end{examples}}
\newcommand{\bdefe}{\begin{defn}}
\newcommand{\edefe}{\end{defn}}
\newcommand{\bprob}{\begin{prob}}
\newcommand{\eprob}{\end{prob}}
\newcommand{\bques}{\begin{ques}}
\newcommand{\eques}{\end{ques}}
\newcommand{\bei}{\begin{itemize}}
\newcommand{\eei}{\end{itemize}}
\newcommand{\bcon}{\begin{conj}}
\newcommand{\econ}{\end{conj}}
\newcommand{\bcons}{\begin{conjs}}
\newcommand{\econs}{\end{conjs}}
\newcommand{\bprop}{\begin{propo}}
\newcommand{\eprop}{\end{propo}}
\newcommand{\br}{\begin{rem}}
\newcommand{\er}{\end{rem}}
\newcommand{\brs}{\begin{rems}}
\newcommand{\ers}{\end{rems}}
\newcommand{\bo}{\begin{obser}}
\newcommand{\eo}{\end{obser}}
\newcommand{\bos}{\begin{obsers}}
\newcommand{\eos}{\end{obsers}}
\newcommand{\bpf}{\begin{proof}}
\newcommand{\epf}{\end{proof}}
\newcommand{\ba}{\begin{array}}
\newcommand{\ea}{\end{array}}
\newcommand{\beq}{\begin{align}}
\newcommand{\beqq}{\begin{align*}}
\newcommand{\eeq}{\end{align}}
\newcommand{\eeqq}{\end{align*}}

\def\cc{\setcounter{equation}{0} 
\setcounter{figure}{0}\setcounter{table}{0}}

\newcounter{minutes}
\divide\time by 60
\newcounter{hours}
\multiply\time by 60 \addtocounter{minutes}{-\time}

\begin{document}

\bibliographystyle{amsplain}
\title[Length of ray images under starlike mappings of order $\alpha$]{A proof of Hall's conjecture on length of ray images under starlike mappings of order $\alpha$}

\thanks{
File:~\jobname .tex,
 printed: \number\day-\number\month-\number\year,
 \thehours.\ifnum\theminutes<10{0}\fi\theminutes}

\author[P. H\"{a}st\"{o}]{Peter H\"{a}st\"{o}}
\address{P. H\"{a}st\"{o},
Department of Mathematics and Statistics, FI-20014,
University of Turku, Turku, Finland
}
\email{pealha@utu.fi}

\author[S. Ponnusamy]{Saminathan Ponnusamy
}
\address{S. Ponnusamy, Department of Mathematics,
Indian Institute of Technology Madras, Chennai-600 036, India}
\email{samy@iitm.ac.in}



\begin{abstract}
Assume that $f$ lies in the class of starlike functions of order $\alpha \in [0,1)$, that is, which
are regular and univalent for $|z|<1$ and such that
\[{\rm Re} \left (\frac{zf'(z)}{f(z)} \right ) > \alpha ~\mbox{ for } |z|<1.
\]
In this paper we show that for each $\alpha \in [0,1)$, the following sharp inequality holds:
\[
|f(re^{i\theta})|^{-1} \int_{0}^{r}|f'(ue^{i\theta})| du \leq \frac{\Gamma (\frac12)\Gamma (2-\alpha )}{\Gamma (\frac32-\alpha )}
~\mbox {for every $r<1$ and $\theta$}.
\]
This settles the conjecture of Hall (1980) positively.
\end{abstract}


\maketitle
\pagestyle{myheadings}
\markboth{P. H\"{a}st\"{o} and S. Ponnusamy}{Length of ray images under starlike mappings of order $\alpha$}
\cc

\section{Introduction and the Main Theorem}
The theory of univalent functions on domains in the complex plane $\IC$ attracted the attention of many for more than
a century, and it has been centered around the class $\mathcal S$ of functions $f$ regular and univalent in the unit
disk $\ID =\{z:\,|z|<1\}$ and normalized by the condition $f(0)=f'(0)-1=0.$ The conjecture of Bieberbach
which asserted $|f^{(n)}(0)/n!|\leq n$ for all $n\geq 2$ (if $f\in {\mathcal S}$), was solved by de Branges \cite{DeB1}
in 1984. The family $\mathcal S$ together with some of its geometric subfamilies play a key role in solving many
extremal problems, and a large amount of research has been done as evidenced by the volume of articles
in the literature (cf. \cite{Du,Gol,Good83,Hay55,PommUniv-75,R1}) and several monographs (cf. \cite{GraKoh2003,Lehto,PommBBCM-92}).
It is still an active field of research in view of several open problems and extensions in several settings \cite{AW-09,HayLing19},
including planar harmonic univalent mappings \cite{CS,DuHar}.

This article concerns length of ray images under a special class of conformal mappings.
Suppose that $f\in {\mathcal F}\subset {\mathcal S}$ and $f$ maps $\ID$ onto a domain $D$. Let $C(r, \theta)$ denote the
image in $D$ of the ray joining $z=0$ to $z=re^{i\theta}\in \ID$ under the mapping $w=f(z)$ belonging to the family ${\mathcal F}$.
Then the length $\ell (r, \theta)$ of the curve $C(r, \theta)$ is given by
\[
\ell (r, \theta):=\int_{0}^{r}|f'(\rho e^{i\theta})|\, d\rho .
\]
In 1963, Gehring and Hayman \cite{GH} showed that if $f\in {\mathcal S}^*\subset {\mathcal S}$, i.e. $f(\ID)$ is starlike (with respect to the origin),
then there exists an absolute constant $M>0$ such that
\be\label{HS-eq0}
\ell (r, \theta)\leq M|f(re^{i\theta})| ~\mbox{ for every $r < 1$ and $\theta$.}
\ee
We refer to this as Gehring--Hayman inequality.
Motivated by this remarkable fact, Sheil-Small \cite{Sh} showed that if $f\in {\mathcal S}^*$, then the constant $M$ in \eqref{HS-eq0}
can be chosen to be $1+\log 4$, and if $ f\in {\mathcal S}^*(\frac12)\subset {\mathcal S}^*   :={\mathcal S}^*(0)$  (see equation \eqref{HP-eq1}),
then the constant may be reduced to $1+\log 2$. Further investigation in this topic led
Sheil-Small \cite{Sh} to conjecture that if $f\in {\mathcal K}\subset {\mathcal S}^*(\frac12)$, i.e $f(\ID)$ is convex, then
the correct constant is $\frac\pi2$. Hall \cite{Ha1,Ha2} showed that the best possible constants are $2$ and $\frac\pi 2$
for the families ${\mathcal S}^*$ and ${\mathcal S}^*(\frac12)$, respectively. This settled both the conjectures of Sheil-Small.
See \cite{BP-96} for a simpler proof
of Gehring--Hayman inequality \eqref{HS-eq0} with $M=2$ for the case of univalent starlike functions.
At this point it is worth recalling the fact that a function belonging to ${\mathcal S}^*(\frac12)$ may not be convex
univalent in $|z|< R$ for any $R >\sqrt{2\sqrt{3}-3}= 0.68$.
It is natural to ask for the corresponding optimal constant $M$ in \eqref{HS-eq0} for several other choices of the family
${\mathcal F}\subset {\mathcal S}$.

In this article, we consider a problem posed by Hall \cite{Ha2}. More precisely, Hall in this paper related the following:
\begin{quote}
{\em
At the Durham Symposium on Analytic Number Theory (July 1979) Professor
Hayman asked in conversation what would be the sharp bound for the class ${\mathcal S}^*(\alpha )$ of
functions starlike of order $\alpha$, that is, which are regular and univalent for $|z| < 1$ and
such that
\be\label{HP-eq1}
{\rm Re}\left( \frac{zf'(z)}{f(z)}\right) >\alpha ~\mbox{ for $|z|<1$}.
\ee
I proved in \cite{Ha1} that in the starlike case, that is when $\alpha =0$, this bound is $2$ (sharp for the Koebe function) and it
is likely that for $0<\alpha <1$ the sharp constant is
\[\frac{\Gamma (\frac12)\Gamma (2-\alpha )}{\Gamma (\frac32-\alpha )}.
\]
From my result for $\alpha =0$, the upper bound $1+ (1-\alpha) (\log 4)^{\alpha}$ can be derived: this is not sharp but numerically it is
pretty good, for example for $\alpha=\frac12$ it gives $1.588\ldots$.
}
\end{quote}

In view of the higher difficulty level of the problem, determining the optimal constant $M$ in \eqref{HS-eq0}
for other choices of the family ${\mathcal F}\subset {\mathcal S}$ is difficult and
results of this type were not available
for many standard geometric subclasses of the univalent family $\mathcal S$.

In the present paper we prove the above conjecture of Hall in full generality for
the class ${\mathcal S}^*(\alpha )$ of functions starlike of order $\alpha$, $0\leq \alpha <1$.
It is worth pointing out that the present method of proof provides also alternate proofs of the two cases,
${\mathcal S}^*(0)$ and ${\mathcal S}^*(\frac12)$, originally settled by Hall \cite{Ha1,Ha2}.


\bthm \label{HS-thm1}
Suppose that $f\in {\mathcal S}^*(\alpha )$, i.e.\ $f$ is a starlike of order $\alpha$ in the unit disk $\ID$. Then
\be\label{HS-eq1}
|f(re^{i\theta})|^{-1} \ell (r, \theta) \leq \beta (\alpha) ~\mbox {for every $r<1$ and $\theta$},
\ee
where $\ell (r, \theta):=\int_{0}^{r}|f'(\rho e^{i\theta})|\, d\rho $ and
\be\label{HS-eq2}
\beta (\alpha ):= \frac{\Gamma (\frac12)\Gamma (2-\alpha )}{\Gamma (\frac32-\alpha )}.
\ee
Furthermore, the constant $\beta(\alpha)$ is optimal.
\ethm

We refer to \cite{BKP, Ka} for some additional research related to Hall's work and conjectures on optimal constants in the Gehring--Hayman inequality.
Related to this problem, we remark that  an attempt has been made by Chen and Ponnusamy \cite{CP2019}
for sense-preserving univalent and $K$-quasiconformal harmonic mappings. In order to make a statement about what this means, we need to introduce some basic notations.

Let $f$ be a complex-valued $C^1$-function defined on  $\ID$ and let
$\ell_{f}(\theta,r)$ be the length of the curve $f|_{[0,z]}$, where $[0, z]$ is a radial line segment from $0$ to
$z=re^{i\theta}\in \ID$, $\theta\in[0,2\pi]$ is fixed and
$r\in[0,1)$. Then (cf. \cite{CLP})
\[\ell_{f}(\theta,r):=\ell\big(f([0, z])\big)=\int_{0}^{r}\left |\,df(\rho
e^{i\theta})\right |=\int_{0}^{r}\left |f_{z}(\rho
e^{i\theta})+e^{-2i\theta}f_{\overline{z}}(\rho e^{i\theta})\right
|\,d\rho.
\]
In \cite{Ke}, Keogh showed that if $f$ is a bounded, analytic and
univalent function in $\mathbb{D}$, then, for each $\theta\in[0,2\pi]$,
\be\label{eq-cpk1}
\ell_{f}(\theta,r)=O \left ( \psi(r)
 \right) ~\mbox{ as $r\rightarrow
1^{-}$},
\ee
where $\psi(r)=\big(\log (1/(1-r) )\big)^{1/2}$  for $0<r<1$, and the exponent $1/2$ in $\psi(r)$ cannot be decreased.
Kennedy \cite{Ken} showed by  examples that
\[\ell_{f}(\theta,r)=O (\mu(r)\psi(r) )~\mbox{as $r\rightarrow 1^{-}$}
\]
is false in general for every positive function $\mu$ in $[0,1)$ satisfying $\mu(r)\rightarrow0$ as
$r\rightarrow1^{-}$. In \cite{CT}, Carroll and Twomey proved this result without the boundedness condition in the following form.

\begin{Thm}\label{Thm-CT}
Suppose that $f(z)=a_{1}z+a_{2}z^{2}+\cdots$ is univalent in
$\mathbb{D}$. Then, for  any fixed $\theta\in[0,2\pi]$, there is a
constant $C_{1}>0$ such that
\be\label{CP-extraeq1}
\ell_{f}(\theta,r)\leq C_{1}\max_{\rho\in[0,r]}|f(\rho e^{i\theta})| \psi(r) 
~\mbox{ for $r\in (0.5,1)$}. \ee If, further, $f(re^{i\theta})=O(1)$
as $r\rightarrow1^{-},$ then {\rm (\ref{eq-cpk1})} holds.
\end{Thm}

Later, Beardon and Carne \cite{BC} gave a relatively simple argument to Theorem~A 
in hyperbolic geometry and provided further examples. Thus, the two works of Hall mentioned in the introduction are sharper versions of
this in the case of functions whose range is either a starlike domain or a convex domain. In spite of the
higher level of difficulty, ideas of \cite{CT,BC} were considered for the class ${\mathcal S}_{H}$ of sense-preserving planar harmonic
univalent mappings $f=h+\overline{g}$ in $\mathbb{D}$, with the normalization $h(0)=g(0)=0$ and $h'(0)=1$
(see \cite{CS,Du}).
The  family ${\mathcal S}_{H}$ together with few geometric subclasses were investigated in \cite{CS,Small}.
For further details, we refer to  \cite{Du,SaRa2013}. If the co-analytic part $g$ is identically zero in the representation $f=h+\overline{g}$, then the class
${\mathcal S}_{H}$ coincides with the family $\mathcal S$. Motivated by the above consideration, in 2019, Chen and Ponnusamy \cite{CP2019}
obtained the following result for the case of planar harmonic mappings.


\begin{Thm}\label{thm-CPx}
For $K\geq1$, let $f\in {\mathcal S}_{H}$ be a $K$-quasiconformal
harmonic mapping. Then, for any fixed $\theta\in[0,2\pi]$, there is
a constant $C_{2}>0$ such that
\[\ell_{f}(\theta,r)\leq C_{2}\max_{\rho\in[0,r]}|f(\rho e^{i\theta})| \psi(r) ~\mbox{ for $r\in (0.5,1)$}.
\]
If, further, $f(re^{i\theta})=O(1)$ as $r\rightarrow1^{-},$ then
\[\ell_{f}(\theta,r)= O (\psi(r))~\mbox{ as $r\rightarrow1^{-}$},
\]
and the exponent $1/2$ in $\psi(r)$  defined above cannot be replaced by a smaller number.
\end{Thm}

First we remark that  Theorem~B   implies Theorem~A when $K=1$.
Secondly, the proof of Theorem~B is relatively harder than the proof
of Theorem~A because of the fact that the arguments of  Beardon and Carne \cite{BC} for Theorem~A
are not applicable in the proof of Theorem~B.

\section{Proof of Theorem \ref{HS-thm1} \label{HS1-20-sec2}}

\subsection{Part 1: Proof of the main Theorem 
}

\blem \label{HS-lem1}
Suppose that $f\in {\mathcal S}^*(\alpha )$. Then the desired inequality \eqref{HS-eq1} holds whenever
\begin{equation}\label{eq:mainClaim}
I(s,t)+I(t,s)<2(\beta(\alpha)-1)\quad\mbox{ for $s,t\in (0,\pi)$ }.
\end{equation}
where
\begin{align}\label{eq11}
I(s,t)&=\int_0^1\left\{ \frac{\sqrt{1+(1-2\alpha)^2u^2+2(1-2\alpha) u\cos t}}{\sqrt{1+u^2-2u\cos
t}} -\frac{1-(1-2\alpha)u^2-2\alpha u\cos t}{1+u^2-2u\cos t}\right\}\times \nonumber\\
&\hspace{3cm}\left\{\frac{2(1-\cos s)}{1+u^2-2u\cos s}\right\}^{1-\alpha}\,du.
\end{align}
\elem
\begin{proof}
The family ${\mathcal S}^*(\alpha )$ is rotationally invariant in the sense that $e^{-i\theta}f(e^{i\theta}z)$ belongs to ${\mathcal S}^*(\alpha )$
whenever $f\in {\mathcal S}^*(\alpha )$. Therefore, without loss of generality, let us suppose that $\theta =0$ in \eqref{HS-eq1}. As a consequence, we
let $h(z)= f(rz)$, $r\in (0,1)$. Then $h$ is regular and univalent for $|z|\le 1$, $h(0)=0$ and  $h(1)= f(r)$.
Therefore to prove \eqref{HS-eq1} we have to show equivalently
that
\begin{equation}\label{eq4}
\int_0^1|h'(u)|\,du\le \beta (\alpha)|h(1)|,
\end{equation}
where $\beta(\alpha)$ is defined by \eqref{HS-eq2}. It remains to show that \eqref{eq4} holds whenever \eqref{eq:mainClaim} holds.

Now, we let $f\in\es^*(\alpha)$. Then, we have
\[H(z):=\frac{zh'(z)}{h(z)} =\frac{rzf'(rz)}{f(rz)}
\quad\text{and}\quad
 {\rm Re}\,H(z)>\alpha,\quad z=re^{i\theta}\in \overline{\ID}.
\]
Using the Herglotz representation theorem for regular functions with positive real part (cf. \cite{Good83,PommUniv-75,R1})
and the fact that $h\in\es^*(\alpha)$, we also have, for $z\in\ID$,
\begin{equation}\label{eq5}
H(z)=\frac{zh'(z)}{h(z)}=\frac{1}{2\pi}\int_{-\pi}^\pi
\frac{1+(1-2\alpha)ze^{-it}}{1-ze^{-it}}\,dV(t),
\end{equation}
where $V(t)$ is an increasing function for
$t\in[-\pi,\pi]$ which satisfies $\frac{V(\pi)-V(-\pi)}{2\pi}=1$.
Therefore, using standard arguments and some computations, we find that
\[ H(u)= \int_{0}^\pi
\frac{1+(1-2\alpha)ue^{-it}}{1-ue^{-it}}\,dW(t),
\]
and
\begin{equation}\label{eq6}
\begin{split}
\frac{\partial}{\partial u}\log|h(u)|
&= u^{-1}{\rm Re}\,H(u)
= \int_0^\pi\frac{1-(1-2\alpha)u^2-2\alpha u\cos t}{u(1+u^2-2u
\cos t )}\,dW(t),
\end{split}
\end{equation}
where $W(t):=\frac{V(t)-V(-t)}{2\pi}$. Note that $W(0)=0$, $W(\pi)=1$ and
$W$ is increasing on $[0,\pi]$ and so $dW$ is nonnegative and has a
total mass $1$. Using (\ref{eq5}) it follows that
\begin{equation}\label{eq7}
\begin{split}
|H(u)|-{\rm Re}\,H(u)
& \le
\int_0^\pi\bigg [\frac{\sqrt{1+(1-2\alpha)^2u^2+2(1-2\alpha) u\cos t}}{\sqrt{1+u^2-2u\cos t}} \\
&\hspace{3cm} -\frac{1-(1-2\alpha)u^2-2\alpha u\cos t}{1+u^2-2u\cos t} \bigg ]\,dW(t) .
\end{split}
\end{equation}

Next we note from the definition of $H(z)$ that
\begin{equation}\label{eq8}
\begin{split}
\int_0^1|h'(u)|\,du
& =  \int_0^1|H(u)|\, |h(u)|u^{-1}\,du\\
& = \int_0^1{\rm Re}\,H(u)|h(u)|u^{-1}\,du+\int_0^1[|H(u)|-{\rm Re}\,H(u)]|h(u)|u^{-1}\,du.
\end{split}
\end{equation}
Regarding the first integral on the right, we find by (\ref{eq6}) that
\begin{equation}\label{eq12}
\int_0^1{\rm Re}\,H(u)|h(u)|u^{-1}\,du= \int_0^1 |h(u)|\frac{\partial}{\partial u}\log|h(u)|\,du   =|h(1)|.
\end{equation}
We then estimate the second of the integrals in (\ref{eq8}). From (\ref{eq6}) we also have
\begin{align*}
\log\left\{\frac{|h(u)|}{|h(1)|}\right\}
& = \int_1^u\frac{\partial}{\partial v}\log |h(v)|\,dv\\
& = \int_1^u v^{-1}{\rm Re}\,H(v)\,dv\\
& =  \int_0^\pi \int_1^u\frac{1-(1-2\alpha)v^2-2\alpha v\cos t}{v(1+v^2-2v\cos t)}\,dv\, dW(t)\\
& =  \int_0^\pi\log\left\{\frac{u(2-2\cos t)^{1-\alpha}}{(1+u^2-2u\cos t)^{1-\alpha}}\right\}\,dW(t).
\end{align*}
Applying Jensen's inequality \cite[p.~24]{Zy} and performing exponentiation on both sides of the last relation, we get
\begin{equation}\label{eq9}
|h(u)|u^{-1} \le |h(1)| \int_0^\pi \left\{\frac{2(1-\cos t)}{1+u^2-2u\cos t}\right\}^{1-\alpha}\,dW(t).
\end{equation}
Therefore, from (\ref{eq7}) and (\ref{eq9}) we deduce that
\begin{equation}\label{eq10}
\begin{split}
\int_0^1\{|H(u)|-{\rm Re}\,H(u)\}|h(u)|u^{-1}\,du
& \le  |h(1)|\int_0^\pi\int_0^\pi I(s,t)\,dW(t)\,dW(s)\\
& \le  \frac{|h(1)|}{2}\int_0^\pi\int_0^\pi[I(s,t)+I(t,s)]\,dW(t)\,dW(s),
\end{split}
\end{equation}
where $I(s,t)$ is given by \eqref{eq11}.

Thus to complete the proof of the inequality (\ref{eq4}), using
(\ref{eq8}), (\ref{eq12}) and (\ref{eq10}), it suffices to show
\be\label{eq10-ex1}
\sup \{I(s,t)+I(t,s):\,0\le t\le \pi,~0\le s\le \pi\}\le 2(\beta(\alpha)-1). \qedhere
\ee
\end{proof}

\subsection{Part 2: Proof of the Inequality \eqref{eq:mainClaim} 
}\label{hp-sec2-2}

To establish the inequality \eqref{eq:mainClaim}, we need to evaluate the integrals $I(t,s)$ and $I(s,t)$,
where $I(t,s)$ is defined by \eqref{eq11}. In order to do this, we rewrite \eqref{eq11} in the following form
\be\label{HS-eq5}
I(s,t)=[2(1-\cos s)]^{1-\alpha}[J(t,s)-K(t,s)],
\ee
where
\[J(s,t)=\int_0^1\frac{ \sqrt{(1+(1-2\alpha)u)^2-2(1-2\alpha) u(1-\cos t)}}{\sqrt{(1+u^2-2u\cos t)}(1+u^2-2u\cos s)^{1-\alpha}}\,du,
\]
and
\[K(s,t)=\int_0^1\frac{[1-(1-2\alpha)u^2-2\alpha u\cos t] }{(1+u^2-2u\cos t)(1+u^2-2u\cos s)^{1-\alpha}} \,du .
\]
In order to prove the inequality \eqref{eq:mainClaim}, we need to establish several lemmas.

Let us denote $S:=2(1-\cos s)$, $T:=2(1-\cos t)$ and $\gamma:=1-2\alpha$ so that $S,T\in (0,4)$ and $\gamma\in (-1,1]$. Then
\eqref{HS-eq5} can be written in terms of $S$ and $T$, which we denote by $I(S,T)$ for obvious reason, and thus, we have
\[
I(S,T) = \int_0^1 \bigg(\frac{S}{(1-u)^2+Su}\bigg)^{\frac{1+\gamma}2}
\bigg[ \frac{\sqrt{(1+\gamma u)^2-\gamma T u}}{\sqrt{(1-u)^2+Tu}} - \frac{1-\gamma u^2-(1-\gamma)(1-\frac T2)u}{(1-u)^2+Tu}\bigg] du.
\]


Our first aim is to give an upper bound for the sum $I(S,T)+ I(T,S)$ in terms of a simpler
integrand. We begin by giving the bound for the
first term in the square bracket factor in the integrand of $I(S,T)$.

\begin{lem}\label{lem:firstTerm}
For $T\in (0,4)$, $\gamma\in (-1,1]$ and $u\in (0,1)$,
\be\label{HS-eq6}
\frac{\sqrt{(1+\gamma u)^2-\gamma T u}}{\sqrt{(1-u)^2+Tu}}
\le
\frac{1+\gamma}2 \frac{1+u}{\sqrt{(1-u)^2+Tu}}
+
\frac{1-\gamma}2 .
\ee
\end{lem}
\begin{proof}
The claim is clear if  $\gamma =1$, so we assume that $\gamma\in (-1,1)$.
As 
\[1+\gamma u=\frac{1+\gamma}2 (1+u) + \frac{1-\gamma}2 (1-u),
\] 
we calculate
\[
\frac{1+\gamma u}{\sqrt{(1-u)^2+Tu}}
=
\frac{1+\gamma}2 \frac{1+u}{\sqrt{(1-u)^2+Tu}} 
+
\frac{1-\gamma}2 \frac{1-u}{\sqrt{(1-u)^2+Tu}} .
\]
Subtracting this from the inequality in the statement of the lemma, we see that the claim \eqref{HS-eq6} is equivalent to
\[
\frac{\sqrt{(1+\gamma u)^2-\gamma T u}-(1+\gamma u)}{\sqrt{(1-u)^2+Tu}}
\le
\frac{1-\gamma}2
\bigg[ 1 -
\frac{1-u}{\sqrt{(1-u)^2+Tu}} \bigg],
\]
or, multiplied by $\frac1{1-\gamma}\sqrt{(1-u)^2+Tu}$,
\be\label{HS-eq7}
\frac1{1-\gamma}\Big[\sqrt{(1+\gamma u)^2-\gamma T u}-(1+\gamma u)\Big]
\le
\frac12
\big[ \sqrt{(1-u)^2+Tu} - (1-u) \big].
\ee
When $\gamma\ge 0$, the left-hand side of \eqref{HS-eq7} is non-positive, so the claim is clear, and therefore, we may assume that $\gamma<0$ and
denote $b:=-\gamma>0$, where $0<b<1$. When $T=0$, both sides equal $0$, so the inequality holds. We may next rewrite \eqref{HS-eq7} equivalently as
$\varphi (T)\geq 0$, where
\[ \varphi (T) = \frac12 \big[ \sqrt{(1-u)^2+Tu} - (1-u) \big]
- \frac1{1+b}\big[\sqrt{(1-b u)^2+bT u}-(1-b u)\big].
\]
We observed that $\varphi (0)=0$ and thus it suffices to show that $\varphi$ is increasing on $(0,4)$. We calculate
\[ \varphi '(T) =
\frac14\,
\frac{u}{ \sqrt{(1-u)^2+Tu}} - \frac1{2(1+b)}\frac{b u}{\sqrt{(1-b u)^2 +bT u}}
\]
and it is non-negative when
\[
\sqrt{(1-b u)^2+b T u} \ge \frac{2b}{1+b}\sqrt{(1-u)^2+Tu}.
\]
Because $1+b\geq 2\sqrt{b}$, the last inequality holds if
\[
\sqrt{(1-b u)^2+b T u} \ge \sqrt{b}\sqrt{(1-u)^2+Tu}.
\]
Squaring both sides gives the equivalent condition
\[
(1-b u)^2+b T u \ge b ((1-u)^2+Tu)
\quad\Leftrightarrow\quad
1-b \ge b(1-b)u^2,
\]
which holds since $b\in (0,1)$ and $u\in (0,1)$. Thus, $\varphi (T)\geq \varphi (0)=0$ and the proof of the lemma is complete.
\end{proof}


\begin{lem}\label{lem:inequality}
Let $a:= \frac TS$, $a\in (0,\infty )$. Then with $S, T\in (0,4)$ and $I(S,T)$ defined as above, we have
\[
I(S,T) + I(T,S)
\le
\frac{1+\gamma}2 \int_0^\infty
\bigg[ (a +w)^{-\frac{1+\gamma}2} + \big(\tfrac1a +w\big)^{-\frac{1+\gamma}2} \bigg]
\frac{\sqrt{1+w}-1}{1+w} w^{-\frac{1-\gamma}2}\, dw,
\]
where $\gamma\in (-1,1]$.
\end{lem}
\begin{proof}
By Lemma~\ref{lem:firstTerm}, we recall that
\[
\frac{\sqrt{(1+\gamma u)^2-\gamma T u}}{\sqrt{(1-u)^2+Tu}}
\le
\frac{1+\gamma}2 \left (\frac{1+u}{\sqrt{(1-u)^2+Tu}} \right )
+
\frac{1-\gamma}2.
\]
For the numerator of the integrand of $K(s,t)$ with the change of variables as in the beginning of Subsection \ref{hp-sec2-2}, i.e. 
precisely the second term of the square-bracked term in the expression of $I(S,T)$,  we use
\[1-\gamma u^2-(1-\gamma)\Big (1-\frac T2\Big)u = \frac{1+\gamma}2 \big (1- u^2 \big) + \frac{1-\gamma}2 \big ((1-u)^2+Tu \big)
\]
so that
\[
\frac{1-\gamma u^2-(1-\gamma)(1-\frac T2)u}{(1-u)^2+Tu}
=
\frac{1+\gamma}2\left (\frac{1- u^2}{(1-u)^2+Tu} \right )+ \frac{1-\gamma}2.
\]
Using these relations, we can therefore estimate
\begin{align*}
&\frac{\sqrt{(1+\gamma u)^2-\gamma T u}}{\sqrt{(1-u)^2+T u}} - \frac{1-\gamma u^2-(1-\gamma)(1-\frac{T}{2})u}{(1-u)^2+Tu} \\
& \qquad \le
\frac{1+\gamma}2 \bigg[\frac{1+u}{\sqrt{(1-u)^2+Tu}} - \frac{1- u^2}{(1-u)^2+Tu} \bigg]
 \\
& \qquad=
\frac{1+\gamma}2 \bigg[\frac{1}{\sqrt{(1-u)^2+Tu}} - \frac{1- u}{(1-u)^2+Tu} \bigg] (1+u).
\end{align*}
Thus, we have established the inequality
\begin{align*}
I(S,T) &\le
\frac{1+\gamma}2 \int_0^1 \bigg(\frac{S}{(1-u)^2+Su}\bigg)^{\frac{1+\gamma}2}
\bigg[\frac{1}{\sqrt{(1-u)^2+Tu}} - \frac{1- u}{(1-u)^2+Tu} \bigg] (1+u)\, du\\
& = \frac{1+\gamma}{2} \int_0^1 \bigg(\frac{S}{1+\frac{Su}{(1-u)^2}}\bigg)^{\frac{1+\gamma}2}
\bigg[\frac{1}{\sqrt{1+ \frac{Tu}{(1-u)^2}}} - \frac{1}{1+ \frac{Tu}{(1-u)^2}} \bigg]  \frac{1+u}{(1-u)^{2+\gamma}}\, du 
\end{align*}
and for $I(T,S)$, it follows similarly that
\begin{align*}
I(S,T) \le
\frac{1+\gamma}{2} \int_0^1 \bigg(\frac{T}{1+\frac{Tu}{(1-u)^2}}\bigg)^{\frac{1+\gamma}2}
\bigg[\frac{1}{\sqrt{1+ \frac{Su}{(1-u)^2}}} - \frac{1}{1+ \frac{Su}{(1-u)^2}} \bigg]  \frac{1+u}{(1-u)^{2+\gamma}}\, du .
\end{align*}
 
Let us continue with the change of variables
\[w:= \frac {Tu} {(1-u)^2}.
\]
Then
\[dw =T \frac{1+u}{(1-u)^3}\, du
\]
so that
\begin{align*}
I(S,T)
&\le
\frac{1+\gamma}2 \int_0^1 \bigg(\frac{S}{1+\frac ST w}\bigg)^{\frac{1+\gamma}2}
\bigg[\frac{1}{\sqrt{1+w}} - \frac{1}{1+w} \bigg] \frac{1+u}{(1-u)^{2+\gamma}}\, du \\
& =
\frac{1+\gamma}2 \int_0^\infty \bigg(\frac{\frac ST}{1+\frac ST w}\bigg)^{\frac{1+\gamma}2}
\frac{\sqrt{1+w}-1}{1+w} [T^{-\frac12}(1-u)]^{1-\gamma}\, dw.
\end{align*}
From the relation $(1-u)^2 = \frac Tw u$, or equivalently the quadratic equation 
\[w(1-u)^2 +T(1-u)-T=0,\]
we solve for $1-u$ with the restriction $0<u<1$:
\[
1-u = \frac12 \left (-\frac Tw + \sqrt{\Big (\frac Tw\Big )^2 +4\frac Tw}\right )
=\frac{\sqrt{T^2+4Tw}-T}{2w}
= \frac{2T}{\sqrt{T^2+4Tw}+T}
\]
so that
\[1-u =\frac{2\sqrt{T}}{\sqrt{T+4w}+\sqrt{T}}\leq \frac{\sqrt{T}}{\sqrt{w}}, ~\mbox{ i.e., }~T^{-\frac12}(1-u) \le \frac1{\sqrt{w}}.
\]
Therefore, we conclude that
\[
I(S,T)
\le
\frac{1+\gamma}2 \int_0^\infty \bigg(\frac{1}{a +w}\bigg)^{\frac{1+\gamma}2}
\frac{\sqrt{1+w}-1}{1+w}  w^{-\frac{1-\gamma}2}\, dw,
\]
where $a= \frac TS \in (0,\infty )$. Interchanging the role of $S$ and $T$ 
in the above proof gives an analogous inequality for $I(T,S)$:
\[
I(T,S)
\le
 \frac{1+\gamma}2 \int_0^\infty \bigg(\frac{1}{b +w}\bigg)^{\frac{1+\gamma}2}
\frac{\sqrt{1+w}-1}{1+w} w^{-\frac{1-\gamma}2}\, dw, \quad b  =\frac{S}{T}  \in (0,\infty).
\]
Note that $b =1/a$.
Finally, adding these two estimates, we obtain the desired claim.
\end{proof}

Let us next consider the expression in the case $\gamma=1$. Based on previous research, it is
already known that the expression in Lemma~\ref{lem:inequality} is maximized when $a=1$. However,
we will need the monotonicity, which is a stronger claim.

\begin{lem}\label{lem:increasing}
The function
\be\label{HS-eq6-ex1}
a\longmapsto
\int_0^\infty \bigg[ \frac1{a+w} + \frac1{\frac1a+w} \bigg] \frac{\sqrt{1+w}-1}{1+w}  \, dw  =:G_1(a)
\ee
is increasing in $(0,1)$.
\end{lem}
\begin{proof}
It turns out that we can explicitly calculate the integrals involved in the expression. Since
\[
\frac1{(a+w)(1+w)} = \frac1{1-a}\bigg[\frac1{a+w} - \frac1{1+w}\bigg],
\]
we calculate
\[
\int_0^\infty \frac{dw}{(a+w)(1+w)} = \left . \frac1{1-a} \ln \left (\frac{a+w}{1+w} \right ) \right |_0^\infty = \frac{\ln \frac1a}{1-a}.
\]
Similarly, 
\[
\int_0^\infty \frac{dw}{\big (\frac1a +w\big )(1+w)} = \frac{\ln a}{1-\frac1a} = \frac{ a \ln \frac1a}{1- a}.
\]

The other two integrals are more complicated, but we find that
\[
\int \frac{dw}{\big (\frac1a +w\big )\sqrt{1+w}} =
2\sqrt{\frac{a}{1-a}} \, \tan^{-1}\bigg(\sqrt{\frac{a}{1-a}}\sqrt{1 + w}\bigg) + C
\]
When we use the formula for the term with $a+w$, the number inside the arctangent is imaginary,
so we use also the formula
\[
\tan^{-1}(z) = \frac i2 \log\bigg(\frac{i+z}{i-z}\bigg).
\]
Hence we conclude
\[
\int \frac{dw}{(a +w)\sqrt{1+w}} =
-\sqrt{\frac{1}{1-a}} \, \log\bigg(\frac{\sqrt{1-a}+\sqrt{1 + w}}{\sqrt{1-a}-\sqrt{1 + w}}\bigg) + C.
\]
With these integral functions, we obtain that
\begin{align*}
&\int_0^\infty \bigg[ \frac1{a+w} + \frac1{\frac1a+w} \bigg] \frac{\sqrt{1+w}-1}{1+w} \, dw
=
-
\frac{1+a}{1-a}\ln\frac1a\\
&\qquad
+
2\sqrt{\frac{a}{1-a}} \bigg[ \frac\pi 2 - \tan^{-1}\bigg(\sqrt{\frac{a}{1-a}}\bigg)\bigg]
-
\sqrt{\frac{1}{1-a}} \bigg[ \log(-1) - \log\bigg(\frac{\sqrt{1-a}+1}{\sqrt{1-a}-1}\bigg) \bigg]
 \\
& =
2\sqrt{\frac{a}{1-a}} \tan^{-1}\bigg(\sqrt{\frac{1-a}{a}}\bigg)
+
\sqrt{\frac{1}{1-a}} \log\bigg(\frac{1+\sqrt{1-a}}{1-\sqrt{1-a}}\bigg)
-
\frac{1+a}{1-a}\ln\frac1a\, .
\end{align*}
The graph of this function, i.e.\ $G_1(a)$, is shown in Figure~\ref{fig:increasing}.

\begin{figure}
\centerline{\includegraphics[width=10cm]{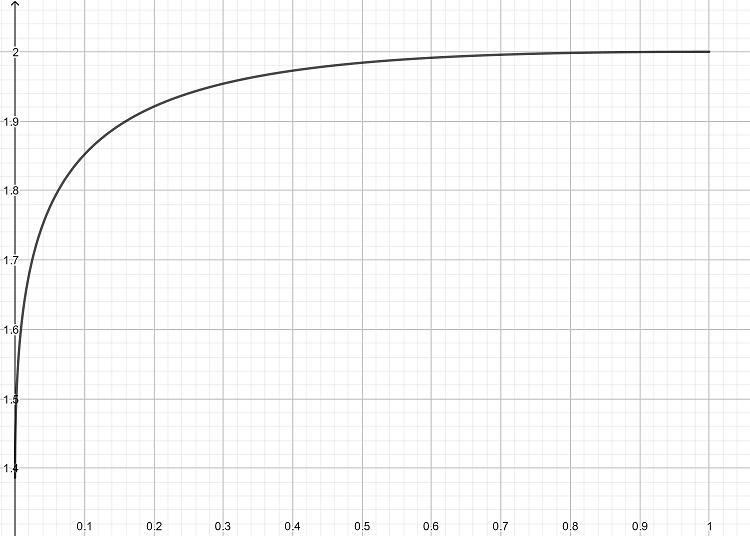}}
\caption{The graph of the function  $G_1(a)$ from Lemma~\ref{lem:increasing}.}\label{fig:increasing}
\end{figure}

We need to show that this expression is increasing in $a$. We change variables by defining
$b:= \sqrt{\frac{1-a}{a}}$ so that $a = \frac1{1+b^2}$ and our expression equals
\[
G (b):=
\frac2b \tan^{-1}(b)
+
2\frac{\sqrt{1+b^2}}b \log(b+\sqrt{1+b^2})
-
\frac{2+b^2}{b^2}\ln (1+b^2 ), ~~b\in (0,\infty ).
\]
Since $b$ is decreasing in $a$, we establish our claim by showing that $G$ is decreasing on $(0,\infty)$.
We calculate
\begin{align*}
\frac12 G'(b)
& =
-\frac1{b^2} \tan^{-1}(b)
+
\frac1{b(1+b^2)}
-
\frac{\log(b+\sqrt{1+b^2})}{b^2\sqrt{1+b^2}}
+
\frac1b
+
\frac{2}{b^3}\ln (1+b^2 )
-
\frac{2+b^2}{b(1+b^2)} \\
&=
-\frac1{b^2} \tan^{-1}(b)
-
\frac{\log(b+\sqrt{1+b^2})}{b^2\sqrt{1+b^2}}
+
\frac{2}{b^3}\ln (1+b^2 ).
\end{align*}
Hence it suffices to show that
\[
g(b):= \frac{b^2}2 G'(b) = -\tan^{-1}(b)
-
\frac{\log(b+\sqrt{1+b^2})}{\sqrt{1+b^2}}
+
\frac{2}{b}\ln (1+b^2 )
\]
is negative. We see that $g(0^+)=0$, and show that $g$ is decreasing on $(0,\infty )$. A calculation gives
\begin{align*}
g'(b) & =  - \frac 1{1+b^2}
+ b\frac{\log(b+\sqrt{1+b^2})}{(1+b^2)^{\frac32}}
- \frac{1}{1+b^2}
- \frac{2}{b^2}\ln (1+b^2 )
+ \frac{4}{1+b^2}\\
&= \frac 2{1+b^2} + b\frac{\log(b+\sqrt{1+b^2})}{(1+b^2)^{\frac32}} - \frac{2}{b^2}\ln (1+b^2 )
\end{align*}
With the new variable $c:=b^2$, we find that
\[
h(c) := c g'(\sqrt{c})
=
\frac{2c}{1+c}
+
\bigg(\frac{c}{1+c}\bigg)^{\frac32}\log(\sqrt{c}+\sqrt{1+c})
-
2\ln (1+c), ~~c\in (0,\infty ).
\]
We need to show that $h$ is negative on $(0,\infty )$, and we observe that $h(0^+)=0$. To show that
$h$ is decreasing on $(0,\infty )$, we calculate the derivative
\begin{align*}
h'(c) &= \frac2{(1+c)^2} +
\frac32 \bigg(\frac{c}{1+c}\bigg)^{\frac12} \frac{\log(\sqrt{c}+\sqrt{1+c})}{(1+c)^2}
+\frac{c}{2(1+c)^2} - \frac 2{1+c}\\
& =-\frac32 \,\frac{c}{(1+c)^2} +
\frac32 \bigg(\frac{c}{1+c}\bigg)^{\frac12} \frac{\log(\sqrt{c}+\sqrt{1+c})}{(1+c)^2},
\end{align*}
from which we define a function $k$ by 
\[
k(c) := \tfrac23 (1+c^2)^{\frac52}c^{-\frac12} h'(c) = -\sqrt{c(1+c)} + \log(\sqrt{c}+\sqrt{1+c}), ~~c\in (0,\infty ).
\]
Finally, we observe that $k(0^+)=0$ and
\[
 k'(c)
 = -\frac{1+2c}{2\sqrt{c(1+c)}} + \frac1{2\sqrt{c(1+c)}}
 =
 -\sqrt{\frac{c}{1+c}} \, \le 0 ~\mbox{ for $c\in (0,\infty )$}.
\]
Thus $k(c)< 0$ for $c\in (0,\infty )$, so that $h$ is decreasing on $(0,\infty )$ and thus negative, which implies that $g$ is decreasing and
negative for $c\in (0,\infty )$. Hence, $G$ is decreasing on $(0,\infty )$, which is equivalent to the original claim.
This completes the proof of the lemma.
\end{proof}

We are now ready to continue our investigation on Lemma~\ref{lem:inequality}, again.

\begin{lem}\label{lem:inequality2}
The maximum of the right-hand side in Lemma~\ref{lem:inequality} is achieved when $a=1$,
so that
\[ I(S,T) + I(T,S)
\le
(1+\gamma) \int_0^\infty \big[(1 +w)^{-\frac{2+\gamma}2} - (1 +w)^{-\frac{3+\gamma}2} \big] w^{-\frac{1-\gamma}2}\, dw,
\]
where $\gamma \in (-1, 1]$.
\end{lem}

\begin{proof}
We consider the function
\be\label{HS-eq6-ex2}
G(a):=
\int_0^\infty
\Big[ (a +w)^{-\frac{1+\gamma}2} + \big(\tfrac1a +w\big)^{-\frac{1+\gamma}2} \Big]
\frac{\sqrt{1+w}-1}{1+w} w^{-\frac{1-\gamma}2}\, dw,
\ee
where $\gamma \in (-1, 1]$ and $ a \in (0,\infty )$. We need to show that $G$ is maximized by $a=1$. To that end, we consider the derivative with respect to
$a$:
\begin{align*}
\frac 2{1+\gamma} G'(a)
&=
\int_0^\infty
\Big[ -(a +w)^{-\frac{3+\gamma}2} + a^{-2}\big(\tfrac1a +w\big)^{-\frac{3+\gamma}2} \Big]
\frac{\sqrt{1+w}-1}{1+w} w^{-\frac{1-\gamma}2}\, dw \\
&=
-\int_0^\infty
\Big(\frac{w}{a +w}\Big)^{\frac{3+\gamma}2}
\frac{\sqrt{1+w}-1}{(1+w)w^2} \, dw
+
\int_0^\infty
\frac1{ a^2}\Big(\frac{w}{\frac1a +w}\Big)^{\frac{3+\gamma}2}
\frac{\sqrt{1+w}-1}{(1+w)w^2}\, dw.
\end{align*}
In the first integral we use the change of variables $v:=\frac wa$ and this gives
\[
\int_0^\infty
\Big(\frac{w}{a +w}\Big)^{\frac{3+\gamma}2}
\frac{\sqrt{1+w}-1}{(1+w)w^2} \, dw
=
\frac1a \int_0^\infty
\Big(\frac{v}{1+v}\Big)^{\frac{3+\gamma}2}
\frac{\sqrt{1+av}-1}{(1+av)v^2}\, dv,
\]
whereas in the second one we use $v:=aw$ and obtain
\[
\int_0^\infty
\frac1{ a^2}\bigg(\frac{w}{\frac1a +w}\bigg)^{\frac{3+\gamma}2}
\frac{\sqrt{1+w}-1}{(1+w)w^2} \, dw
=
\frac1a\int_0^\infty
\bigg(\frac{v}{1+v}\bigg)^{\frac{3+\gamma}2}
\frac{\sqrt{1+\frac va}-1}{(1+\frac va)v^2}\, dv.
\]
Therefore, we have the following expression for the derivative
\[
\frac {2a}{1+\gamma} G'(a) =
\int_0^\infty
\Big(\frac{v}{1+v}\Big)^{\frac{3+\gamma}2}\frac{1}{v^2}
\bigg[\frac{\sqrt{1+\frac va}-1}{1+\frac va}-\frac{\sqrt{1+av}-1}{1+av} \bigg]\, dv.
\]
Denote
\[g(x):= x^{-\frac12}-x^{-1}
\]
and observe that the square bracket term equals $g(1+\frac va)-g(1+va)$.
We find that
\[g'(x) = -\tfrac12 x^{-\frac32}+x^{-2} = \tfrac12 (2-\sqrt x)x^{-2}
\]
so that $g$ is increasing on $[0,\sqrt2]$ and decreasing on $[\sqrt2,\infty)$. When $a<1$, we have $1+\frac va>1+av$ and so
it follows that $v\mapsto g(1+\frac va)-g(1+va)$ is positive until some value $v_0$ and then negative.
Furthermore, the function
\[v\mapsto \left (\frac{v}{1+v}\right )^{-\frac{1-\gamma}2}
\]
is decreasing on $(0,\infty )$. Therefore, we have
\[
[g(1+\tfrac va)-g(1+va)] \Big(\frac{v}{1+v}\Big)^{-\frac{1-\gamma}2}
\ge 
[g(1+\tfrac va)-g(1+va)] \Big(\frac{v_0}{1+v_0}\Big)^{-\frac{1-\gamma}2}
\]
both when $v\le v_0$ and $v\ge v_0$. We conclude that
\be\label{HS-eq6-ex3}
\frac {2a}{1+\gamma} G'(a)
\ge
\Big(\frac{v_0}{1+v_0}\Big)^{-\frac{1-\gamma}2}
\int_0^\infty
\Big(\frac{v}{1+v}\Big)^{2}\frac{1}{v^2}
\bigg[\frac{\sqrt{1+\frac va}-1}{1+\frac va}-\frac{\sqrt{1+av}-1}{1+av} \bigg]\, dv.
\ee
Up to a constant, the right-hand side of \eqref{HS-eq6-ex3} is the derivative of the function in the case $\gamma=1$.
By Lemma~\ref{lem:increasing}, this function is increasing on $(0,1)$, so its derivative,
and hence the right-hand side of the inequality in \eqref{HS-eq6-ex3} above, is non-negative. It follows that
$G' (a)\ge 0$ on $(0,1)$. Furthermore, by symmetry we conclude that $G' (a)\le 0$ on $(1,\infty)$.
Hence the maximum of $G$ occurs at $a=1$, as claimed.
\end{proof}


Finally, we are ready to prove that the inequality \eqref{eq:mainClaim} holds when $\alpha \in [0, 1)$.

\subsection{Proof of Theorem \ref{HS-thm1}}
 To prove \eqref{eq10-ex1} it suffice to show that
\[\sup \{I(S,T) + I(T,S):\,0\le S\le 4,~0\le T\le 4\}\le 2(\beta(\alpha)-1). 
\]
Recall that $\gamma=1-2\alpha$ and so, by Lemma~\ref{lem:inequality2}, to suffices to show equivalently that
\[ (1-\alpha) \int_0^\infty \big[(1+w)^{\alpha-\frac{3}2} - (1+w)^{\alpha-2} \big] w^{-\alpha}\, dw
\le  \frac{\Gamma (\frac12)\Gamma (2-\alpha )}{\Gamma (\frac32-\alpha )} -1,
\]
by the definition of $\beta (\alpha)$.
We then consider the beta function (not the function $\beta(\alpha)$ from before), and its relation to the gamma function as follows
\[ 
\int_0^\infty t^{x-1}(1+t)^{-x-y}\, dt = \frac{\Gamma(x)\Gamma(y)}{\Gamma(x+y)}.
\]
We use this formula with $x=1-\alpha$ and $y=\frac12$ or $y=1$. This gives that
\begin{align*}
(1-\alpha) \int_0^\infty \big[(1+w)^{\alpha-\frac{3}2} - (1+w)^{\alpha-2} \big] w^{-\alpha}\, dw 
&=
(1-\alpha) \bigg[\frac{\Gamma(\frac12)\Gamma(1-\alpha)}{\Gamma(\frac32-\alpha)} - \frac{\Gamma(1)\Gamma(1-\alpha)}{\Gamma(2-\alpha)} \bigg] \\
&= \frac{\Gamma(\frac12)\Gamma(2-\alpha)}{\Gamma(\frac32-\alpha)} - 1,
\end{align*}
since $(1-\alpha)\Gamma(1-\alpha) = \Gamma(2-\alpha)$ and $\Gamma(1)=1$.
This completes the proof of the desired estimate \eqref{eq:mainClaim}, which, by Lemma~\ref{HS-lem1} implies
that the Gehring--Hayman inequality holds with constant $\beta(\alpha)$.

It remains to be shown that $\beta(\alpha)$ given by \eqref{HS-eq2} cannot be replaced
by any smaller constant.  We show that the extremal function for our problem is $k_{\alpha}$ defined by
$k_{\alpha}(z):=z/(1-z)^{2-2\alpha}$. We calculate that
\[
k_{\alpha}'(z)=\frac{1+(1-2\alpha)z}{(1-z)^{3-2\alpha}}
\quad\text{and}\quad
\frac{z k_\alpha'(z)}{k_\alpha(z)} = \alpha + (1-\alpha)\frac{1+z}{1-z}.
\]
From this we see that $k_{\alpha} \in {\mathcal S}^*(\alpha ) $.
As before, we set $\gamma=1-2\alpha$ and we have that
\[
|k_\alpha'(re^{i\theta})| = \frac{\sqrt{1+(\gamma r)^2+2\gamma r\cos \theta}}{(1+r^2-2r\cos\theta)^{1+\frac\gamma2}}\, .
\]
Furthermore, $|k_\alpha(e^{i\theta})| = (2(1-\cos \theta))^{-\frac{1+\gamma}2}$. Let us denote again
$T:=2(1-\cos \theta)$. Then we have shown that
\[
\lim_{r\to 1} \frac{\ell(r,\theta)}{|k_\alpha(re^{i\theta})|}
= T^{\frac{1+\gamma}2}
\int_0^1 \frac{\sqrt{(1+\gamma u)^2-\gamma  T u}}{((1-u)^2+Tu)^{1+\frac\gamma2}}\, du\, .
\]
We are interested in the limit value of the right-hand side when $T\to 0$. For some small $\epsilon>0$, we restrict the integral to
the range $u\in (1-\epsilon,1)$ for a lower bound, and estimate
\[
\sqrt{(1+\gamma u)^2-\gamma  T u}
\ge
(1+\gamma) (1- O(\epsilon+T)).
\]
We estimate the remaining terms with the same change of variables $w:=\frac {Tu}{(1-u)^2}$ as before:
\[
\int_{1-\epsilon}^1 \frac{T^{\frac{1+\gamma}2}}{((1-u)^2+Tu)^{1+\frac\gamma2}}\, du
\ge
\int_{T(1-\epsilon)/\epsilon^2}^\infty \frac{T^{\frac{-1+\gamma}2}(1-u)^{1-\gamma}}{(2-\epsilon)(1+w)^{1+\frac\gamma2}}\, dw.
\]
Also as before, solving for $T^{-\frac12}(1-u)$ and using $T\le \epsilon w$ (which follows from $u\ge 1-\epsilon$), we have:
\[
T^{-\frac12}(1-u)
=
\frac2{\sqrt{T+4w}+\sqrt{T}}
\ge
\frac1{\sqrt{\epsilon/4+1}+\sqrt{\epsilon/4}} \frac1{\sqrt{w}}
=(1-O(\sqrt\epsilon))\frac1{\sqrt{w}}.
\]
With the previous two estimates, we obtain
\[
\lim_{T\to 0} \int_{1-\epsilon}^1 \frac{T^{\frac{1+\gamma}2}\sqrt{(1+\gamma u)^2-\gamma  T u}}{((1-u)^2+Tu)^{1+\frac\gamma2}}\, du
\ge
(1-O(\sqrt\epsilon))\frac{1+\gamma}2
\int_0^\infty w^{-\frac{1-\gamma}2}(1+w)^{-1-\frac\gamma2}\, dw
\]
so that
\[
\lim_{\theta\to 0} \lim_{r\to 1} \frac{\ell(r,\theta)}{|k_\alpha(re^{i\theta})|}
\ge (1-O(\sqrt\epsilon)) (1-\alpha) \int_0^\infty w^{-\alpha}(1+w)^{\alpha - \frac32}\, dw
=
(1-O(\sqrt\epsilon)) \frac{\Gamma(\frac12)\Gamma(2-\alpha)}{\Gamma(\frac32-\alpha)}\, .
\]
The claim follows from this as $\epsilon \to 0$.
\hfill $\Box$

\bigskip

\subsection*{Acknowledgements}
This work was completed during the visit of the second author to the University of Turku, Finland in 2019-2020, and this author thanks Prof. Matti Vuorinen
for his continuous support and his encouragement. The visit of this author was supported by a grant under ``India-Finland Joint
Call for Mobility of Researchers Programme" and this author acknowledges the support received from the Department of Science \& Technology, Government of India,
and the Academy of Finland.

\vspace{.4cm}
\noindent
{\bf Compliance with ethical standards:} 

\vspace{.1cm}

\noindent
{\bf Conflict of interest.} The authors declare that they do not have conflict of interests.

\vspace{.1cm}
\noindent
{\bf Ethical standards.} The research complies with ethical standards.



\begin{thebibliography}{99}

%
%

\bibitem{AW-09}
F. G. Avkhadiev and K.-J. Wirths,
\textit{Schwarz-Pick type inequalities},
Frontiers in Mathematics, Birkhäuser Verlag, Basel, 2009. viii + 156 pp.

\bibitem{BKP} {\rm R. Balasubramanian, V. Karunakaran and S. Ponnusamy,}
A proof of Hall's conjecture on starlike mappings,
\emph{J. London Math. Soc.} \textbf{48}(2) (1993), 271--282.

\bibitem{BP-96} {\rm R. Balasubramanian and S. Ponnusamy,}
An alternate proof of Hall's theorem on a conformal mappings inequality,
\emph{Bull. Belgium Math. Soc.} \textbf{3} (1996), 209--213.

\bibitem{BC} {\rm A. F. Beardon and  T. K. Carne,} Euclidean and hyperbolic
lengths of images of arcs, \emph{Proc. London Math. Soc.,} {\bf
97} (2008), 183--208.

\bibitem{CT}
{\rm T. Carroll and  J. B. Twomey,} Conformal mappings of close-to-convex domains,
\emph{J. London Math. Soc.,} {\bf 55} (1997), 489--498.

\bibitem{DeB1} {\rm L. de Branges,}
A proof of the Bieberbach conjecture,
\emph{Acta Math.} \textbf{154} (1985), 137--152.

 \bibitem{CLP} {\rm S. L. Chen, G. Liu and  S. Ponnusamy,}
Linear measure and $K$-quasiconformal harmonic mappings (in
Chinese), \emph{Sci. Sin. Math.,} {\bf 47} (2017), 1--10.

\bibitem{CP2019} {\rm Sl. Chen and S. Ponnusamy},
{Radial length, radial John disks and $K$-quasiconformal harmonic mappings,}
\emph{Potential Analysis} \textbf{50}(2019), 415--437.

\bibitem{CS} J. G. Clunie and T. Sheil-Small,
Harmonic univalent functions,
\emph{Ann. Acad. Sci. Fenn. Ser. A. I.} \textbf{9} (1984), 3--25.


\bibitem{Du} P. L. Duren,
\textit{Univalent Functions}, Springer-Verlag, New York, (1983).

\bibitem{DuHar} P. Duren,
\textit{Harmonic mappings in the plane},
 Cambridge university Press, New York, (2004).


\bibitem{GH} {\rm F. W. Gehring and W. K. Hayman},
An inequality in the theory of conformal mapping,
\emph{J. Math. Pures Appl.} \textbf{41} (1963), 353--361.

\bibitem{Gol} G. M. Goluzin,
\textit{Geometric Theory of Functions of a Complex Variable},
Translations of Mathematical Monographs, Vol. 26, AMS, Providence (1969).

\bibitem{Good83} A. W. Goodman,
\textit{Univalent Functions}, Vols. I, II, Tampa, FL: Mariner (1983).

\bibitem{GraKoh2003} I. Graham and G. Kohr,
\textit{Geometric Function Theory in One and Higher Dimensions},
New York: Marcel Dekker Inc. (2003)

\bibitem{Ha1} {\rm R. R. Hall},
 The length of ray-images under starlike mappings,
\emph{Mathematika} \textbf{23} (1976), 147--150.

\bibitem{Ha2} {\rm R. R. Hall}, A conformal mapping inequality for starlike functions of order $\frac12$,
\emph{Bull. London Math. Soc.} \textbf{12} (1980), 119--126.

\bibitem{Hay55} W. K. Hayman,
\textit{Multivalent Functions}, Cambridge Tracts in Mathematics 110,
2nd ed. Cambridge: Cambridge University Press (1994).

\bibitem{HayLing19} W. K. Hayman and E. F. Lingham,
\textit{Research problems in function theory},
Fiftieth anniversary edition of Problem Books in Mathematics, Springer, Cham, 2019. viii+284 pp.

\bibitem{Ka} {\rm V. Karunakaran},
Length of ray-images under conformal maps,
\emph{Proc. Amer. Math. Soc.} \textbf{87} (1983), 289--294.

%

\bibitem{Ken} {\rm P. B. Kennedy,}
Conformal mapping of bounded domains, \emph{J. London Math. Soc.,}
{\bf 31} (1956), 332--336.

\bibitem{Ke} {\rm F. R. Keogh,} A property of bounded schlicht functions,
\emph{J. London Math. Soc.,} {\bf 29} (1954), 379--382.

%

\bibitem{Lehto} {\rm O. Lehto,}
\textit{Univalent functions and Teichm\"{u}ller spaces}, Graduate Texts in Math. 109,
Springer-Verlag, NewYork (1987).


\bibitem{PommUniv-75} Ch. Pommerenke,
\textit{Univalent Functions}, with a chapter on quadratic differentials
by Gerd Jensen. Studia Mathematica/Mathematische Lehrb\"{u}cher 15.
G\"{o}ttingen: Vandenhoeck and Ruprecht (1975).

\bibitem{PommBBCM-92} Ch. Pommerenke,
\textit{Boundary Behaviour of Conformal Maps},
Springer, New York (1992)


\bibitem{SaRa2013}  {\sc S. Ponnusamy and A. Rasila},
\textrm{Planar harmonic and quasiregular mappings},
Topics in Modern  Function Theory (Editors. St. Ruscheweyh and S. Ponnusamy): Chapter in
CMFT, RMS-Lecture Notes Series No. 19, 2013, pp. 267--333.

\bibitem{Rob36} {\rm Malcolm I. S. Robertson},
On the theory of univalent functions,
\emph{Ann. of Math.} \textbf{37}(2) (1936), 3749--408.


\bibitem{R1} {\rm St. Ruscheweyh}, \textit{Convolutions in geometric function theory},
Les Presses de l'Universit\'e de Montr\'eal, Montr\'eal (1982).

\bibitem{Sh} {\rm T. Sheil-Small},
Some conformal mapping inequalities for starlike and convex functions,
\emph{J. London Math. Soc.} \textbf{1}(2) (1969), 577--587.

\bibitem{Small} {\rm T. Sheil-Small,}
Constants for planar harmonic mappings,
\emph{J. London Math. Soc.,} {\bf 42} (1990), 237--248.


\bibitem{Zy} {\rm A. Zygmund}, Trigonometric series, Cambridge Univ. Press,
London and NewYork, 1968.
\end{thebibliography}
\end{document}